\documentclass[11pt]{article}
\usepackage[misc]{ifsym} 
\usepackage{lipsum}
\usepackage{amssymb}
\usepackage{amsmath}
\usepackage{latexsym}
\usepackage{bm}
\usepackage{enumitem}
\usepackage{graphicx}
\usepackage{amscd}
\usepackage{extarrows}
\usepackage{amsfonts}
\usepackage{amssymb}
\usepackage{indentfirst}
\usepackage{bbm}
\usepackage{mathdots}
\usepackage{mathtools}
\usepackage[linktocpage]{hyperref}
\usepackage{amsthm}
\hypersetup{
    colorlinks,
    citecolor=black,
    filecolor=black,
    linkcolor=blue,
    urlcolor=black
}

\usepackage{quiver}
\usepackage[multiple]{footmisc}
\usepackage[title]{appendix}

\newcommand\blfootnote[1]{%
  \begingroup
  \renewcommand\thefootnote{}\footnote{#1}%
  \addtocounter{footnote}{-1}%
  \endgroup
}

\usepackage{float}

\theoremstyle{plain}
\newtheorem{theorem}{Theorem}[section]
\newtheorem{proposition}[theorem]{Proposition}
\newtheorem{lemma}[theorem]{Lemma}
\newtheorem{corollary}[theorem]{Corollary}
\newtheorem*{theorem*}{Theorem}

\theoremstyle{remark}
\newtheorem{remark}[theorem]{Remark}

\theoremstyle{definition}

\newtheorem{example}[theorem]{Example}

\DeclareMathOperator{\spn}{span}

\DeclareMathOperator{\tr}{tr}

\DeclareMathOperator{\covol}{covol}
\DeclareMathOperator{\Tr}{Tr}

\DeclareMathOperator{\id}{id}
\DeclareMathOperator{\Ind}{Ind}

\DeclareMathOperator{\rk}{rank}

\DeclareMathOperator{\SL}{SL}
\DeclareMathOperator{\GL}{GL}
\DeclareMathOperator{\PSL}{PSL}
\DeclareMathOperator{\SO}{SO}
\DeclareMathOperator{\Li}{Li}

\title{Plancherel Measures of Reductive Adelic Groups and von Neumann Dimensions}

\author{Jun Yang}

\date{}

\begin{document}
\maketitle


\blfootnote{\Letter~Jun Yang~~\href{mailto:junyang@fas.harvard.edu}{junyang@fas.harvard.edu}}
\blfootnote{Harvard University, Cambridge, MA 02138, USA}
\blfootnote{This work was supported in part by the ARO Grant W911NF-19-1-0302 and the ARO MURI Grant W911NF-20-1-0082.}

\begin{abstract}
Given a number field $F$ and a reductive group $G$ over $F$, 
the unitary dual $\widehat{G(\mathbb{A}_F)}$ of the adelic group $G(\mathbb{A}_F)$ and the Placherel measure $\nu_{G(\mathbb{A}_F)}$ on it can be determined by the Plancherel measure of its local groups $G(F_v)$. 
Given a subset $X\subset \widehat{G(\mathbb{A}_F)}$ of finite Plancherel measure, let $H_X$ be the direct integral of the irreducible representations in $X$. 
Besides a $G(\mathbb{A}_F)$-module and a $G(F)$-module, $H_X$ is also a module over the group von Neumann algebra $\mathcal{L}(G(F))$, hence there is a canonical dimension $\dim_{\mathcal{L}(G(F))}H_X\in [0,\infty)$. 
It is proved that the Plancherel measure of $G(\mathbb{A}_F)$ coincides with the dimension over the algebra $\mathcal{L}(G(F))$: 
\begin{center}
$\dim_{\mathcal{L}(G(F))}H_X=\nu_{G(\mathbb{A}_F)}(X)$,
\end{center}
if $G$ is semisimple, simply connected and $G(\mathbb{A}_F)$ is equipped with the Tamagawa measure. 

\end{abstract}
\tableofcontents

\section{Introduction}

Let $F$ be a number field and $G$ be a reductive algebraic group over $F$.  The unitary irreducible representations of the adelic group $G(\mathbb{A}_F)$ are of great interest as the automorphic representation side of the global Langlands correspondence. 
It is then natural to consider the Plancherel measure on the equivalence classes of these representations, i.e., the unitary dual $\widehat{G(\mathbb{A}_F)}$. 
It can be shown that the dual space  $\widehat{G(\mathbb{A}_F)}$ with the Plancherel measure $\nu_{G(\mathbb{A}_F)}$ is the restricted product of its local dual $\widehat{G(F_v)}$ with respect to the unramified representations (see Theorem \ref{tPm}).  

Meanwhile, the study of the von Neumann algebra $\mathcal{L}(\Gamma)$ of a countable discrete group $\Gamma$ is one of the most challenging topics in operator algebras. 
Very little is known of how $\mathcal{L}(\Gamma)$ depends on the group $\Gamma$. 
Especially, a well known conjecture asks whether $\mathcal{L}(\Gamma_1)\cong \mathcal{L}(\Gamma_2)$ implies $\Gamma_1\cong \Gamma_2$ provided $
\Gamma_1,\Gamma_2$ are free groups or lattices in a real Lie group \cite{J00ten}.  
In this paper, we consider the group von Neumann algebra $\mathcal{L}(G(F))$ of the countable group $G(F)$ defined over the number field $F$.  
The modules over $\mathcal{L}(G(F))$ can be described by a canonical dimension function $\dim_{\mathcal{L}(G(F))}$, which determines the isomorphic class of an $\mathcal{L}(G(F))$-module if $G(F)$ is an infinite conjugacy class (ICC) group (or equivalently, $\mathcal{L}(G(F))$ is a factor of type $\text{II}_1$).  

We show the coincidence of the Plancherel measures $\nu_{G(\mathbb{A}_F)}$ of the adelic group $G(\mathbb{A}_F)$ and the dimensions over the von Neumann algebra $\mathcal{L}(G(F))$ of the group $G(F)$.  
More precisely, let $X$ be a subset of $\widehat{G(\mathbb{A}_F)}$ with a finite Plancherel measure, i.e., $\nu_{G(\mathbb{A}_F)}(X)<\infty$. 
Denoting the underlying Hilbert space of an irreducible representation $\pi$ by $H_{\pi}$, we let $H_X$ be the direct integral of the  spaces $H_{\pi}$'s such that the isomorphism class $[\pi] \in X$, which is  $H_X=\int_X^{\oplus} H_{\pi}d\nu_{G(\mathbb{A}_F)}([\pi])$. 
The following result (Theorem \ref{tmain1}) is proved : 
\begin{theorem*}
Let $G$ be a simply connected semisimple algebraic group defined over a number field $F$. 
Then, for a subset $X$ of $\widehat{G(\mathbb{A}_F)}$ with finite Plancherel measure, we have  
\begin{center}
    $\dim_{\mathcal{L}(G(F))}H_X=\nu_{G(\mathbb{A}_F)}(X)$. 
\end{center}
\end{theorem*}
In another word, one can tell the Placherel measure of the adelic group $G(\mathbb{A}_F)$ just by treating the representations as modules over the von Neumann algebra $\mathcal{L}(G(F))$ of the discrete subgroup $G(F)$. 
In Section \ref{sPadele}, we give an explicit description of the Plancherel measure of an adelic group. 
In Section \ref{svndim}, we construct the Hilbert space from the direct integral of irreducible representations of $G(\mathbb{A}_F)$ and discuss its von Neumann over $\mathcal{L}(G(F))$. 
In Section \ref{sprod}, we reach a product formula involving the covolume of a lattice.
An example is also provided. 
In Section \ref{smain}, we prove the theorem above. 

{\it Acknowledgements} 
The author would like to thank V. Jones for encouraging me to work on these problems. 
The author is grateful for the detailed comments from D. Bisch, F. Radulescu, and many fruitful conversations with A. Jaffe, B. Mazur, D. Vogan, and W. Schmid.

\section{The Plancherel measure of an adelic group}\label{sPadele}

Let $G$ be a locally compact Hausdorff group with a Haar measure $\mu$. 
Let $\widehat{G}$ be the unitary dual of $G$, i.e., the set of equivalence classes $[\pi]$ of irreducible unitary representations $(\pi,H_{\pi})$ of $G$. 
For convenience, we will not distinguish a representation $\pi$ and its equivalence class $[\pi]$
In the following sections, we will let $\pi$ denote its equivalence class $[\pi]$ or also a representative in it.

We first give a brief review of type I groups following \cite{DiCalg,Fo2}. 
A unitary representation $(\pi,H)$ of $G$ is called {\it primary} if the von Neumann algebra $\mathcal{A}(\pi)\subset B(H)$ generated by $\{\pi(g)|g\in G\}$ is a factor, i.e., its center $Z(A(\pi))$ are scalar multiples of the identity. 
If $\pi$ is a direct sum of irreducible representations, then $\pi$ is primary if and only if all these irreducible summands are unitary equivalent. 

A group $G$ is called {\it type I} if every primary representation of $G$ is a direct sum of copies of some irreducible representation. 
That is to say, $\mathcal{A}(\pi)$ is a type I factor for every primary representation $\pi$.  
It is known that compact groups, connected semisimple Lie groups, connected real algebraic groups, 
reductive $p$-adic groups, and
also the adelic group $G(\mathbb{A})$ of a connected reductive group $G$ defined over a number field are type I (see \cite{Bnst74,Clz07,Fo2,Kiri76}). 

Let us consider the regular representations of a type I locally compact group $G$ and the Fourier transform. 
Let $\rho,\lambda$ be the right and left regular representation of $G$ on $L^2(G)$:
\begin{center}
$(\rho(g)f)(x)=R_{g}f(x)=f(xg)$, $(\lambda(g)f)(x)=L_{g}f(x)=f(g^{-1}x)$
\end{center}
for $f\in L^2(G)$. 
This gives us the {\it two sided regular representation} $\tau$ of $G\times G$ on $L^2(G)$: 
$(\tau(g,h)f)(x)=R_gL_hf(x)=L_hR_gf(x)=f(h^{-1}xg)$. 
If $f\in L^1(G)$, the {\it Fourier transform} of $f$ is defined to be the measurable field of operators over $\widehat{G}$ given by
\begin{center}
$\hat{f}(\pi)=\int_G f(x)\pi(x^{-1})d\mu(x)$,
\end{center}
with a representative of a equivalence class $[\pi]\in \widehat{G}$.
We have the convolution $\widehat{f_1 * f_2}=\widehat{f_1}\cdot\widehat{f_2}$ \cite{Fo2,Kiri04}.  
Now we let 
\begin{center}
$\mathcal{J}^1=L^1(G)\cap L^2(G)$, $\mathcal{J}^2=\spn\{f*g|f,g\in \mathcal{J}^1\}$. 
\end{center}
For a unitary representation $(\pi,H_{\pi})$ of $G$, we denote its contragradient by $(\overline{\pi},H_{\overline{\pi}})$. 
We let $\Tr$ denote the usual trace defined on the trace-class operators. 

\begin{theorem}[The Plancherel Theorem]\label{tplancherel}
Suppose $G$ is a second-countable, unimodular, type I group. 
There is a measure $\nu$ on $\widehat{G}$, uniquely determined once the Haar measure $\mu$ on $G$ is fixed, with the following properties. 
\begin{enumerate}
    \item The Fourier transform $f\mapsto \hat{f}$ maps $\mathcal{J}^1$ into $\int_{G}^{\oplus}H_{\pi}\otimes H_{\overline{\pi}}d\nu(\pi)$ and it extends to a unitary map from $L^2(G)$ onto  $\int_{G}^{\oplus}H_{\pi}\otimes H_{\overline{\pi}}d\nu(\pi)$ that intertwines the two-sided regular representation $\tau$ with $\int_{G}^{\oplus}{\pi}\otimes {\overline{\pi}}d\nu(\pi)$.
    \item For $f_1,f_2\in \mathcal{J}^1$, one has the Parseval formula
    \begin{center}
        $\int_G f_1(x)\overline{f_2(x)}d\mu(x)=\int_{\widehat{G}}\Tr[\hat{f_1}(\pi)\hat{f_2}(\pi)^*]d\nu(\pi)$. 
    \end{center}
    \item For $f\in \mathcal{J}^2$, one has the Fourier inversion formula
    \begin{center}
        $f(x)=\int_{\widehat{G}}\Tr[{\pi}(x)^*\hat{f}(\pi)]d\nu(\pi)$, $x\in G$. 
    \end{center}
\end{enumerate}
\end{theorem}
Indeed, the Plancherel measure can be determined by either the second or third property described above.  
This theorem also implies the decomposition of right and left regular representation of $G$:
\begin{center}
$\rho\cong \int_{\widehat{G}}^{\oplus}\pi\otimes {\id_{\overline{H_\pi}}}~d\nu(\pi)$, $\lambda\cong \int_{\widehat{G}}^{\oplus}{\id_{H_\pi}}\otimes\overline{\pi}~d\nu(\pi)$. 
\end{center}
The measure $\nu$ above is called the {\it Plancherel measure} on $\widehat{G}$. 
Note if $G$ is abelian, the Plancherel measure is a Haar measure;  
if $G$ is compact, the Plancherel measure is given by $\nu(E)=\sum_{\pi \in E}\dim_{\mathbb{C}}{H_\pi}$ for $E\subset \widehat{G}$. 
The support of $\nu$ is not always all of $\widehat{G}$. 
Indeed, ${\rm supp}(\nu)=\widehat{G}$ if and only if $G$ is amenable (see \cite{DiCalg}).

\begin{remark}\label{rtenHS}
For a Hilbert space $H$, let $H^*$ be its dual space, which is isomorphic to the conjugate space $\overline{H}$.
We may identify the tensor Hilbert space $H\otimes H^*$ with the Hilbert-Schmidt operators on $H$, denoted by $B(H)_{\rm HS}$, which is equipped with the inner product $\langle x,y\rangle_{B(H)_{\rm HS}}=\Tr(xy^*)$.  
Taking an orthonormal basis $\{e_i\}_{i\geq 1}$ of $H$, there is an isometric isomorphism $\Psi\colon B(H)_{\rm HS}\to H\otimes H^*$ given as 
\begin{center}
$\Psi(T)=\sum_{i,j}\langle Te_j,e_i\rangle\cdot e_i\otimes e_j^*$, $T \in B(H)_{\rm HS}$, 
\end{center}
such that $\Tr(TT^*)=\|\Psi(T)\|_{H\otimes H^*}^2$. 
\end{remark}

Now let $F$ be a number field and $G$ be a reductive linear algebraic group defined over $F$.
We let $V$ ($V_f$ and $V_{\infty}$, respectively) denote the set of equivalence classes of places (finite places and infinite places, respectively) of $F$ and $F_v$ denotes the local fields at $v\in V$. 

Let $G_v=G(F_v)$ and  $G(\mathbb{A}_F)=\prod_{v\in V_{\infty}}G(F_v)\times \prod'_{v\in V_f}G(F_v)$, the group over the adele $\mathbb{A}_F$ of $F$. 
For each $v\in V_f$, let $K_v$ be a special open maximal compact subgroup of $G_v$.
Indeed, for almost all $v$, $K_v$ can be taken to be $G(\mathcal{O}_v)$, the integral points of $G_v$. 
For each finite place $v\in V_f$, we fix a Haar measure $\mu_v$ on the local group $G_v=G(F_v)$.
We may assume it is normalized in the sense that $\mu_v(K_v)=1$. 
Let $\nu_v$ be the Planacherel measure on $\widehat{G_v}$ determined by $\mu_v$. 

Denote an irreducible unitary representation of $G_v$ by $(\pi_v,H_{\pi_v})$, which will also stand for the isomorphism class in the unitary dual $\widehat{G_v}$. 
Hence we have $L^2(G_v)=\int_{\widehat{G_v}}H_{\pi_{v}}\otimes H_{\pi_v}^* d\nu_v(\pi_v)$ by the Plancherel Theorem \ref{tplancherel}.
Consider the characteristic function $\chi_v$ of $K_v$, i.e.,
\begin{equation*}
    \chi_v(x)=
    \begin{cases}
      1, & \text{if}\ x\in K_v; \\
      0, & \text{otherwise}. 
    \end{cases}
 \end{equation*}
Observe $\|\chi_v\|_{L^2(G_v)}=\mu_v(K_v)=1$.  

Now we take an irreducible representation $(\pi_v,H_v)$ from $\widehat{G_v}$ and consider the Fourier transform $\widehat{\chi_v}$ of $\chi_{v}$ at $\pi_v$, which is given by
\begin{center}
$\widehat{\chi_v}(\pi_{v})=\int_{x\in G_v}\chi_v(x)\pi_{v}(x^{-1})d\mu_v(x)$. 
\end{center}
Please note $\widehat{\chi_v}(\pi_{v})$ is a Hilbert-Schmidt operator, or equivalently, a vector in $H_{\pi_{v}}\otimes H_{\overline{\pi_{v}}}$ almost everywhere on $\widehat{G_v}$ (see \ref{rtenHS}).  
\begin{lemma}\label{lunivec}
If $\widehat{\chi_v}(\pi_{v})$ is nonzero, we have
$\|\widehat{\chi_v}(\pi_{v})\|_{H_{\pi_{v}}\otimes H_{\overline{\pi_{v}}}}=1$. 
\end{lemma}
\begin{proof}
Since the group $K_v$ is compact, the restriction  $\pi_{v}|_{K_v}$ can be written as a direct sum of irreducible subrepresentations of $K_v$, i.e.,
$\pi_{v}|_{K_v}=\bigoplus_{i\in I}\rho_i$. 
Then the Fourier transform of the characteristic function $\chi_v$ can be obtained as follows: 
\begin{equation*}
\begin{aligned}
\widehat{\chi_v}(\pi_{v})&=\int_{G_v}\chi_v(x)\pi_{v}(x^{-1})d\mu_v(x)=\int_{K_v}\pi_{v}(x^{-1})d\mu_v(x)\\
&=\bigoplus_{i\in I}\int_{x\in K_v}\rho_i(x^{-1})d\mu_v(x).
\end{aligned}
\end{equation*}
If $\rho_i$ is a nontrivial irreducible representation, we have $\int_{ K_v}\rho(x^{-1})d\mu_v(x)=0$. 
But for the trivial representation ${\rm 1}_{K_v}$, $\int_{ K_v}{\rm 1}_{K_v}(x^{-1})d\mu_v(x)=1$. 
Therefore $\widehat{\chi_v}(\pi_{v})$ is the multiplicity of ${\rm 1}_{K_v}$ in $\pi_v$.  

It is well-known that an irreducible unitary representation of $G(F_v)$ contains ${\rm 1}_{K_v}$ at most once (which are the unramified ones, see \cite{Flath77} Theorem 2). 
Hence $\widehat{\chi_v}(\pi_{v})$ is $0$ or a projection on $B(H_{\pi_{v}})$ with rank $1$, or equivalently, the zero vector or a unit vector in the tensor space $H_{\pi_{v}}\otimes H_{\overline{\pi_{v}}}$.  
\end{proof}

Let $X_v\subset \widehat{G_v}$ be the subset of unramified representations of $G_v$.
\begin{corollary}\label{cmeassph}
We have $\nu_v(X_v)=1$, the Plancherel measure of unramified representations is $1$. 
\end{corollary}
\begin{proof}
It follows the fact that
\begin{center}
$1=\|\chi_v\|_{L^2(G_v)}^2=\int_{\widehat{G_v}}\|\widehat{\chi_v}(\pi_{v})\|^2d\nu_v(\pi_{v})$, 
\end{center}
where $\|\widehat{\chi_v}(\pi_{v})\|=1$ for unramified representations and $\|\widehat{\chi_v}(\pi_{v})\|=0$ otherwise. 
\end{proof}

\begin{remark}[The restricted product of measure spaces \cite{Blkd77}]
Let $\{(X_i,\mu_i):i\in I\}$ be a set of measure spaces with subsets $Y_i\subset X_i$ such that $\mu_i(Y_i)=1$ for all but finitely many $i\in I$.  
Write $(X_S,\nu_S)=(\prod_{i\in S}X_i)\times (\prod_{i\notin F}Y_i)$ for a finite subset $S\subset I$.  
Note for $S'\subset S$, there is a natural embedding of measure spaces $X_{S'}\subset X_S$.
The {\it restricted product of $\{(X_i,\mu_i):i\in I\}$ with respect to $\{Y_i:i\in I\}$} is defined by
\begin{center}
$\prod'_{i\in I} (X_i,Y_i,\mu_i)=\varinjlim_S X_S=\{(x_i)_{i\in I}\in \prod X_i|x_i\in Y_i\text{~for~almost~all~}i\}$,
\end{center}
which will also be denoted simply by $(X,\nu)$. 
A subset $Z\subset X$ is measurable if and only if $Z\cap X_S$ is measurable for all finite $S$ and its measure is given by $\mu(S)=\sup_{S}\nu_S(Z\cap X_S)$. 
\end{remark}

Now we take the restricted product of the measures $\{\mu_v\}_{v\in V}$ as a Haar measure on the adelic group $G(\mathbb{A})$, denoted by $\mu_{G_{\mathbb{A}}}$, or simply $\mu$. 
Let $\nu_{G(\mathbb{A})}$ be the Plancherel measure of $\widehat{G_{\mathbb{A}}}$ determined by $\mu$.  
\begin{theorem}\label{tPm}
$(\widehat{G_{\mathbb{A}}},\nu_{G(\mathbb{A})})=\prod'_{v\in V}(\widehat{G_v},X_v,\nu_v)$.
\end{theorem}
\begin{proof}
We have $L^2(G_{\mathbb{A}},\mu_{\mathbb{A}})=\bigotimes'_{v\in V}(L^2(G_v),\chi_v)$ by \cite{Blkd77} Corollary 2.3, where the restriced product is taken with respect to the unit vectors $\widehat{\chi_v}$.  
This can be rewritten as 
\begin{center}
    $L^2(G_{\mathbb{A}})=\bigotimes'_v\left(\int_{\widehat{G_v}}H_{\pi_{v}}\otimes H_{\overline{\pi_{v}}}d\nu_v(\pi_{v})\right)$. 
\end{center}
As shown in \cite{Flath77}, an irreducible representation $\pi$ of $G(\mathbb{A})$ can be written as $\pi\cong \bigotimes \pi_{v}$, where each $\pi_v$ is an irreducible unitary representation of $G_v$. 
Almost all $\pi_v$ are unramified, i.e., $\pi_v\in X_v$ for all but finitely many $v$.  

Take a finite set $S\subset V$ and consider the function $h_S=(\otimes_{v\in S}~h_v)\otimes (\otimes_{v\in V\backslash S}~\chi_{v})$ with $h_v\in \mathcal{J}^2_{G_v}$. 
For any $x=(x_v)_{v\in V}\in G_{\mathbb{A}}$, we let $S_x=\{v \in V|x_v\notin K_v\}$, the finite set of the places $v$ whose local coordinate $x_v$ is outside of the compact group $K_v$.
Then the Fourier inversion formula of $h$ at $x$ is given by
{\small
\begin{equation}\label{eadeleplan1}
\begin{aligned}
h((x_v)_{v\in V})
=&\int_{\widehat{G_{\mathbb{A}}}}\Tr(\widehat{h}(\pi))\pi((x_v)_{v\in V})^{*}d\nu_{\mathbb{A}}(\pi)\\
=&\int_{\widehat{G_{\mathbb{A}}}}\left(\prod_{v\in S}\Tr(\widehat{h_v}(\pi_v))\pi_v(x_v)^{*})\cdot \prod_{v\in V\backslash S}\Tr(\widehat{\chi_{K_v}}(\pi_v)\pi_v(x_v)^{*})\right)d\nu_{\mathbb{A}}(\otimes \pi_v)\\
=&\int_{\widehat{G_{\mathbb{A}}}}\left(\prod_{v\in S}\Tr(\widehat{h_v}(\pi_v))\pi_v(x_v)^{*})\cdot \prod_{v\in S_x\backslash S}\Tr(\widehat{\chi_{K_v}}(\pi_v)\pi_v(x_v)^{*})\right)\cdot \\
&~~~~~~\left(\prod_{v\in V\backslash (S\cup S_x)}\delta_{X_v}(\pi_v)\right)d\nu_{\mathbb{A}}(\otimes'\pi_v)
\end{aligned}
\end{equation}
}

Observing that for $x_v\in K_v$, $\widehat{\chi_{v}}(\pi_v)\circ\pi_v(x_v)=\int_{K_v}\pi_v(y^{-1}x_v)d\mu_v(y)$ depends only on $\pi_v$,  
we have $\delta_{X_v}(\pi_v)=1$ or $0$ depends on whether $\pi_v$ is unramified or not. 
On the other hand, we have
{\small
\begin{equation}\label{eadeleplan2}
\begin{aligned}
h((x_v)_{v\in V})
=&\prod_{v\in S}h_v(x_v)\cdot \prod_{v\in V\backslash S}\chi_{K_v}(x_v)\\
=&\prod_{v\in S}\int_{\widehat{G_v}}\Tr(\widehat{h_v}(\pi_v)\pi_v(x_v)^{*})d\nu_v(\pi_v)\cdot \prod_{v\in V\backslash S}\int_{\widehat{G_v}}\Tr(\widehat{\chi_{K_v}}(\pi_v)\pi_v(x_v)^{*})d\nu_v(\pi_v)\\
=&\prod_{v\in S}\int_{\widehat{G_v}}\Tr(\widehat{h_v}(\pi_v)\pi_v(x_v)^{*})d\nu_v(\pi_v)\cdot \prod_{v\in S_x\backslash S}\int_{\widehat{G_v}}\Tr(\widehat{\chi_{K_v}}(\pi_v)\pi_v(x_v)^{*})d\nu_v(\pi_v)\\
&~~~~~~\cdot \prod_{v\in V\backslash (S\cup S_x)}\int_{\widehat{G_v}}\delta_{X_v}(\pi_v)d\nu_v(\pi_v),\\
\end{aligned}
\end{equation}
}
where the last product is $1$ by Corollary \ref{cmeassph}. 

Observe the Plancherel measure is uniquely determined by the Fourier inversion formula.
The functions of the same form as $h$ above span a dense subset of the $L^2$-space. 
Comparing the equation (\ref{eadeleplan1}) with (\ref{eadeleplan2}), we conclude $(\widehat{G_{\mathbb{A}}},\nu_{\mathbb{A}})=\prod'(\widehat{G_v},X_v,\nu_v)$. 
\end{proof}

Note an irreducible unitary representation $\pi=\otimes \pi_v$ is a discrete series representation (an irreducible direct summand in the $L^2$-space) if and only if $\nu_{\mathbb{A}}(\{\pi\})>0$. 
But $\nu_v(\pi_v)=0$ for almost all $v$. 
Hence $\nu_{\mathbb{A}}(\pi)=\prod'\nu_{p}(\pi_v)=0$ and we obtain the following known result: 
\begin{corollary}
$G(\mathbb{A})$ has no discrete series representations. 
\end{corollary}

\section{The von Neumann dimensions of direct integrals}\label{svndim}

Let $\Gamma$ be a countable group with the counting measure.  
Let $\{\delta_{\gamma}\}_{\gamma\in \Gamma}$ be the usual orthonormal basis of $l^2(\Gamma)$. 
We also let $\lambda$ and $\rho$ be the left and right regular representations of $\Gamma$ on $l^2(\Gamma)$ respectively.
For all $\gamma,\gamma'\in \Gamma$, we have
$\lambda(\gamma')\delta_{\gamma}=\delta_{\gamma'\gamma}$ and $\rho(\gamma')\delta_{\gamma}=\delta_{\gamma\gamma'^{-1}}$. 
Let $\mathcal{L}(\Gamma)$ be the strong operator closure of the complex linear span of $\lambda(\gamma)$'s (or equivalently,  $\rho(\gamma)$'s). 
This is the {\it group von Neumann algebra of $\Gamma$}. 
There is a canonical faithful normal tracial state $\tau_{\Gamma}$, or simply $\tau$, on $\mathcal{L}(\Gamma)$, which is given by
\begin{center}
$\tau(x)=\langle x\delta_e,\delta_e\rangle_{l^2(\Gamma)}$, $x\in \mathcal{L}(\Gamma)$. 
\end{center}
Hence $\mathcal{L}(\Gamma)$ is a finite von Neumann algebra (hence it is of type $\text{I}$ or $\text{II}_1$).  

More generally, for a tracial von Neumann algebra $M$ with the trace $\tau$, we consider the GNS representation of $M$ on the Hilbert space constructed from the completion of $M$ with respect to the inner product $\langle x,y\rangle_{\tau}=\tau(xy^*)$. 
The underlying space will be denoted by $L^2(M,\tau)$, or simply $L^2(M)$.
 
Consider a normal unital representation $\pi\colon M\to B(H)$ with both $M$ and $H$ separable.  
There exists an isometry $u\colon H\to L^2(M)\otimes l^2(\mathbb{N})$, which commutes with the actions of $M$:
\begin{center}
$u\circ\pi(x)=(\lambda(x)\otimes\id_{l^2(\mathbb{N})} )\circ u$, $\forall x\in M$,
\end{center}
where $\lambda\colon M\mapsto L^2(M)$ denotes the left action. 
Then $p=uu^*$ is a projection in $B(L^2(M)\otimes l^2(\mathbb{N}))$ such that $H\cong p( L^2(M)\otimes l^2(\mathbb{N}))$. 
We have the following result (see \cite{APintrII1} Chapter 8). 

\begin{proposition}\label{ptrdim}
The correspondence $H\mapsto p$ above defines a bijection between the set of equivalence classes of left $M$-modules and the set of equivalence classes of projections in $(M'\cap B(L^2(M)))\otimes B(l^2(\mathbb{N}))$. 
\end{proposition}

The {\it von Neumann dimension} of the $M$-module $H$ are defined to be $(\tau\otimes \Tr)(p)$ and denoted by $\dim_M(H)$, which takes its value in $[0,\infty]$. 
We have: 
\begin{enumerate}
    \item $\dim_M(\oplus_i H_i)=\sum_i \dim_M(H_i)$. 
    \item $\dim_M(L^2(M))=1$.
\end{enumerate} 
Note $\dim_M(H)$ depends on the trace $\tau$. 
If $M$ is a finite factor, i.e., $Z(M)\cong\mathbb{C}$, there is a unique normal tracial state (see \cite{J83,MvN36}) and we further have: 
\begin{enumerate}
\setcounter{enumi}{2}
    \item $\dim_M(H)=\dim_M(H')$ if and only if $H$ and $H'$ are isomorphic as $M$-modules (provided $M$ is a factor).  
\end{enumerate}
When $M$ is not a factor, there is a $Z(M)$-valued trace which determines the isomorphism class of an $M$-module (see \cite{Bek04}). 

In the following sections, we will consider the group von Neumann algebra $\mathcal{L}(\Gamma)$ with the canonical trace $tr(x)=\langle x\delta_e,\delta_e \rangle$. 
Hence the von Neumann dimension of $\mathcal{L}(\Gamma)$ is the one uniquely determined by this trace. 
Note a discrete group $\Gamma$ is called an infinite conjugacy class (ICC) group if every nontrivial conjugacy class $C_{\gamma}=\{g\gamma g^{-1}|g\in \Gamma\}$, $\gamma\neq e$, is infinite. 
It is well-known that $\mathcal{L}(\Gamma)$ is a $\rm{II}_1$ factor if and only if $\Gamma$ is a nontrivial ICC group.

Now we consider the case that $\Gamma$ is a discrete subgroup of a locally compact unimodular type I group $G$.  
Let $\mu$ be a Haar measure of $G$. A measurable set $D\subset G$ is called a {\it fundamental domain} for $\Gamma$ if $D$ satisfies $\mu(G\backslash \cup_{\gamma\in\Gamma}\gamma D)=0$ and 
$\mu(\gamma_1 D\cap \gamma_2 D)=0$ if $\gamma_1\neq \gamma_2$ in $\Gamma$. 
In this section, we always assume $\Gamma$ is a lattice, i.e., $\mu(D)<\infty$.
The measure $\mu(D)$ is called {\it covolume} of $\Gamma$ and will be denoted by $\covol(\Gamma)$. 
Note the covolume depends on the Haar measure $\mu$ (see Remark \ref{rcovol}). 

There is a natural isomorphism $L^2(G)\cong l^2(\Gamma)\otimes L^2(D,\mu)$ given by
\begin{center}
$\phi\mapsto \sum_{\gamma\in\Gamma}\delta_{\gamma}\otimes \phi_{\gamma}$ with $\phi_{\gamma}(z)=\phi(\gamma\cdot z)$,
\end{center}
where $z\in D$ and $\gamma\in \Gamma$. 
The restriction representation $\rho_G|_{\Gamma}$ of $\Gamma$ is the tensor product of $\rho_{\Gamma}$ on $l^2(\Gamma)$ and the identity operator $\id$ on $L^2(D,\mu)$. 
Hence the von Neumann algebra $\rho_G(\Gamma)''\cong \mathcal{L}(\Gamma)\otimes \mathbb{C}=\mathcal{L}(\Gamma)$, which will be denoted by $M$ throughout this section. 
Please note $L^2(M)=l^2(\Gamma)$. 


Suppose $X$ is a measurable subset of 
$\widehat{G}$ with the Plancherel measure $\nu(X)<\infty$. 
Define
\begin{center}
$H_X=\int_{X}^{\oplus}H_{\pi}d\nu(\pi)$,   
\end{center} 
which is the direct integral of the spaces $H_{\pi}$ with $\pi\in X$.  
Suppose $\{e_k(\pi)\}_{k\geq 1}$ is an orthonormal basis of its underlying Hilbert space $H_{\pi}$. 
We have the following natural $G$-equivariant isometric isomorphism from $H_X$ to a subspace of $L^2(G)$: 
\begin{equation}\label{eHXsub}
\begin{aligned}
H_X~~~~~~~~~&\cong ~~~\int_{X}^{\oplus}H_{\pi}\otimes e_1(\pi)^*d\nu(\pi)\\
v=\int_{X}^{\oplus}v(\pi)d\nu(\pi)~~&\mapsto~~~ \int_{X}v(\pi)\otimes e_1(\pi)d\nu(\pi),
\end{aligned}
\end{equation}
Therefore we will not distinguish these two spaces and denote them both by $H_X$. 

Consider the projection $P_X\colon L^2(G)\to H_X$ defined on a dense subspace of $L^2(G)$ as follows:  
\begin{center}
{\small
$\int_{\widehat{G}}^{\oplus}\left(\sum\limits_{i,j\geq 1}a_{i,j}(\pi)e_j(\pi)\otimes e_i(\pi)^*\right)d\nu(\pi)\mapsto \int_{\widehat{X}}^{\oplus}\left(\sum\limits_{j\geq 1}a_{1,j}(\pi)e_j(\pi)\otimes e_1(\pi)^*\right)d\nu(\pi)$
},
\end{center}
where all but finite $a_{i,j}(\pi)\in \mathbb{C}$ are zero for each $\pi$.  
Observe that $P_X$ commutes with the right regular representation $\rho$, i.e., $\rho\circ P_X=P_X\circ \rho$. 
We have a unitary representation of $G$ on $H_X$ and denote it by $\rho_X$. 

Given two vectors $v=\int_{X}^{\oplus}v(\pi)d\nu(\pi)$ and $w=\int_{X}^{\oplus}w(\pi)d\nu(\pi)$ in $H_X$ with $v(\pi),w(\pi)\in H_{\pi}$, we have $v(\pi)\otimes w(\pi)^{*}\in H_{\pi}\otimes H_{\pi}^*$.
The corresponding Hilbert-Schmidt operator will also be denoted by $v(\pi)\otimes w(\pi)^{*}$ (see Remark \ref{rtenHS}).
We define a function on $G$ by
\begin{center}
$C_{v,w}(g)=\langle \rho_X(g^{-1})v,w\rangle_{H_X}=\int_{X}\Tr(\pi(g)^{*}v(\pi)\otimes w(\pi)^*)d\nu(\pi)$. 
\end{center}
\begin{lemma}
For $v,w\in H_X$, we have $C_{v,w}\in L^2(G)$. 
Moreover, 
\begin{center}
$\langle C_{v_1,w_1},C_{v_2,w_2}\rangle_{L^2(G)}=\int_{X}\Tr(v_1(\pi)\otimes w_1(\pi)^{*}\cdot w_2(\pi)\otimes v_2(\pi)^{*})d\nu(\pi)$, 
\end{center}
for any vectors $v_1,v_2,w_1,w_2\in H_X$. 
\end{lemma}
\begin{proof}
Since $v_i(\pi)\otimes w_i(\pi)^*$ is Hilbert-Schmidt, $v_1(\pi)\otimes w_1(\pi)^{*}\cdot w_2(\pi)\otimes v_2(\pi)^{*}$ is a trace-class operator.
The integral is bounded by the norm of $v,w$ in $H_X$ and hence well-defined. 
As $v(\pi)=w(\pi)=0$ for $\pi \in \widehat{G}- X$, we have
\begin{center}
    $C_{v,w}(g)=\int_{X}\Tr(\pi(g)^*v(\pi)\otimes w(\pi)^*)d\nu(\pi)=\int_{\widehat{G}}\Tr(\pi(g)^*v(\pi)\otimes w(\pi)^*)d\nu(\pi)$. 
\end{center}
By Theorem \ref{tplancherel}, we conclude that $v\otimes w^*=\int_{X}v(\pi)\otimes w(\pi)^*d\nu(\pi)=\widehat{C_{v,w}}$, the Fourier transform of the function $C_{v,w}$ on $G$.  
Thus we obtain
\begin{equation*}
\begin{aligned}
\langle C_{v_1,w_1},C_{v_2,w_2}\rangle_{L^2(G)}&=\int_{\widehat{G}}\Tr \left(\widehat{C_{v_1,w_1}}(\pi)\widehat{C_{v_1,w_1}}(\pi)^*\right)d\nu(\pi)\\
&=\int_{X}\Tr \left(\widehat{C_{v_1,w_1}}(\pi)\widehat{C_{v_1,w_1}}(\pi)^*\right)d\nu(\pi)\\
&=\int_{\widehat{X}}\Tr \left(v_1(\pi)\otimes w_1(\pi)^{*}\cdot w_2(\pi)\otimes v_2(\pi)^{*}\right)d\nu(\pi)\\
\end{aligned}
\end{equation*}
Then it is easy to get $C_{v,w}\in L^2(G)$.
\end{proof}

Let $\{d_k\}_{k\geq 1}$ be an orthonormal basis of $L^2(D)$. 
Note the restricted representation $\rho_X|_{\Gamma}$ makes the Hilbert space $H_X$ an $M$-module, whose von Neumann dimension can be obtained as follows. 
\begin{lemma}\label{ldimsumnorm}
With the assumption above, we have
\begin{center}
    $\dim_M(H_X)=\sum_{k\geq 1}\|Pd_k\|_{H_X}^2$.
\end{center}
\end{lemma}
\begin{proof}
Let $u$ be the inclusion $H_X\to L^2(G)$. 
We have $u^*u={\rm id}_{H_X}$ and $uu^*=P_X$. 
Note $L^2(G)\cong L^2(M)\otimes L^2(D,dg)$, where $L^2(M)$ is the standard $M$-module and $L^2(D,dg)$ is regarded as a trivial $M$-module. 
Thus, by definition (see Proposition \ref{ptrdim}), we know
\begin{center}
$\dim_M(H_X)=\Tr_{M'\cap B(L^2(G))}(P)$, 
\end{center}
where $M'\cap B(L^2(G))=\{T\in B(L^2(G))|Tx=xT,~\forall x\in M\}$, the commutant of $M$ on $L^2(G)$.
On the right-hand side, 
\begin{center}
 $\Tr_{M'\cap B(L^2(G))}=\tr_{M'\cap B(L^2(M))}\otimes \Tr_{B(L^2(D))}$    
\end{center}
is the natural trace on $M'$. 

The commutant $M'$ is generated by the finite sums of the form
\begin{center}
$x=\sum_{\gamma\in \Gamma}\rho_{\gamma}\otimes a_{\gamma}$,
\end{center}
where $\rho_{\gamma}=J\lambda(\gamma)J\in M'\cap L^2(M)$ (here $J\colon L^2(M)\to L^2(M)$ is the conjugate linear isometry extended from $x\mapsto x^*$) and $a_{\gamma}$ is a finite rank operator in $B(L^2(D))$. 

Let $d_m^*\otimes d_n$ denotes the operator $\xi \mapsto \langle d_m,\xi\rangle\cdot d_n$ on $L^2(D)$. 
Then each $a_{\gamma}$ can be written as $a_{\gamma}=\sum_{m,n\geq 1}a_{\gamma,m,n}d_m^*\otimes d_n$ with $a_{\gamma,m,n}\in \mathbb{C}$ and all but finite terms of $a_{m,n}$ are trivial. 
Thus we obtain
\begin{center}
$\Tr_{M'}(\rho_{\gamma}\otimes a_{\gamma})=\tr_M(\lambda_{\gamma})\sum_{m\geq 1}a_{\gamma,m,m}=\delta_{e}(\lambda)\Tr_{L^2(D)}(a_\gamma)$.
\end{center}
This is equivalent to say
\begin{center}
$\Tr_{M'}(x)=\Tr_{L^2(D)}(a_e)$. 
\end{center}

Let $Q$ be the projection of $L^2(G)$ onto $L^2(D)\cong \mathbb{C}\delta_e\otimes L^2(D)$. 
Then $\Tr_{L^2(D)}=\Tr_{L^2(G)}(QyQ)$ for $y\in B(L^2(D))$. 
We have
\begin{equation}\label{eTrQxQ}
    \Tr_{M'}(x)=\Tr_{L^2(D)}(a_e)=\Tr_{L^2(G)}(Qa_eQ)=\Tr_{L^2(G)}(QxQ)
\end{equation}
As $P$ is a strong limit of elements that have the same form as $x$ above and the traces are normal, formula (\ref{eTrQxQ}) holds for $x=P$ and we obtain
\begin{equation*}
\begin{aligned}
\dim_M(H_X)&=\Tr_{M'}(P)=\Tr_{L^2(G,dg)}(QPQ)\\
&=\sum\nolimits_{k\geq 1}\langle QPQd_k,d_k\rangle_{L^2(G)}=\sum\nolimits_{k\geq 1}\langle Qd_k,PQd_k\rangle_{L^2(G)}\\
&=\sum\nolimits_{k\geq 1}\langle d_k,Pd_k\rangle_{L^2(G)}=\sum\nolimits_{k\geq 1}\langle Pd_k,Pd_k\rangle_{L^2(G)}\\
&=\sum\nolimits_{k\geq 1}\langle Pd_k,Pd_k\rangle_{H_X}=\sum\nolimits_{n\geq 1}\|Pd_k\|_{H_X}^2
\end{aligned}
\end{equation*}
\end{proof}

\section{A product formula}\label{sprod}
\begin{theorem}\label{tdimmeas}
Let $G$ be a locally compact unimodular type I group with Haar measure $\mu$. 
Let $\nu$ be the Plancherel measure on the unitary dual $\widehat{G}$ of $G$. 
Suppose $\Gamma$ is a lattice in $G$ and $\mathcal{L}(\Gamma)$ is the group von Neumann algebra of $\Gamma$.  
Let $X\subset\widehat{G}$ such that $\nu(X)<\infty$ and $H_X=\int_X^{\oplus} H_{\pi}d\nu(\pi)$.  
We have 
\begin{center}
$\dim_{\mathcal{L}(\Gamma)}(H_X)=\covol(\Gamma)\cdot \nu(X)$. 
\end{center}
\end{theorem}

\begin{proof}
We take a vector $\eta=\int_{X}^{\oplus}\eta(\pi)d\nu(\pi)$ in $H_X$ such that $\|\eta(\pi)\|_{H_{\pi}}^2=\frac{1}{\nu(X)}$ almost everywhere in $X$. 
Then $\eta$ is a unit vector in $H_X$ and also in $L^2(G)$. 

Observe $\{\delta_{\gamma}\otimes d_k\}_{\gamma\in \Gamma,n\geq 1}$ is an orthogonal basis of $L^2(G,\mu)$ via the isomorphism $L^2(G)\cong l^2(\Gamma)\otimes L^2(D,\mu)$. 
Note $\{d_k\}_{k\geq 1}$ can be regarded as functions in $L^2(G)$ with support in $D$.  We have $\delta_{\gamma}\otimes d_k=\rho_G(\gamma)d_k$ for $k\geq 1$ and $\gamma\in \Gamma$. 
Hence we have
\begin{center}
$1=\|\rho_X(g)\eta\|_{H_X}^2=\|\rho(g)\eta\|_{L^2(G)}^2=\sum_{\gamma\in \Gamma,k\geq 1}|\langle \rho(g)\eta,\rho(\lambda)d_k\rangle_{L^2(G)}|^2$. 
\end{center}

Consequently, we obtain: 
\begin{equation*}
\begin{aligned}
{\rm covol}(\Gamma)&=\int_{D}1d\mu(g)=\int_{D}\sum_{\gamma\in \Gamma,k\geq 1}|\langle \rho(\lambda)^{*}\rho(g)\eta,d_k\rangle|^2d\mu(g)\\
&=\sum_{k\geq 1}\int_{G}|\langle P\rho(g)\eta,d_k\rangle_{L^2(G)}|^2d\mu(g)=\sum_{k\geq 1}\int_{G}|\langle \rho(g)\eta,Pd_k\rangle_{H_X}|^2d\mu(g)\\
&=\sum_{k\geq 1}\int_{G}|\langle \rho(g)^*Pd_k,\eta\rangle|^2d\mu(g)=\sum_{k\geq 1}\langle C_{Pd_k,\eta},C_{Pd_k,\eta}\rangle_{L^2(G)}\\
&=\sum_{k\geq 1}\int_X \Tr( (Pd_n)(\pi)\otimes \eta(\pi)^*\cdot (\eta(\pi)\otimes(Pd_k)(\pi)^*)d\nu(\pi)\\
&=\sum_{k\geq 1}\int_X \langle  (Pd_n)(\pi)\otimes \eta(\pi)^*,(Pd_n)(\pi)\otimes \eta(\pi)^*\rangle_{H_{\pi}\otimes H_{\pi}^*}d\nu(\pi)\\
&=\sum_{k\geq 1}\int_{X}\|\eta(\pi)\|_{H_\pi}^2\cdot \|(Pd_k)(\pi)\|_{H_{\pi}}^2d\nu(\pi)\\
&=\frac{1}{\nu(X)}\sum_{n\geq 1}\|Pd_k\|_{H_X}^2,
\end{aligned}
\end{equation*}
which is $\dim_M(H_X)\cdot \nu(X)^{-1}$ by Lemma \ref{ldimsumnorm}. 
Hence we get $\dim_M(H_X)=\covol(\Gamma)\cdot \nu(X)$. 
\end{proof}
\begin{remark}\label{rcovol}
\begin{enumerate}
\item If $\mu'=k\cdot \mu$ is another Haar measure on $G$ for some $k>0$, the covolumes are related by  $\covol'(\Gamma)=\mu'(G/\Gamma)=k'\cdot\mu(G/\Gamma)=k\cdot\covol(\Gamma)$.
But the induced Plancherel measure $\nu'=k^{-1} \cdot\nu$ and the dependencies cancel out in the formula above. 
    
\item The conditions of $G$ above are satisfied for all connected semisimple Lie groups and connected reductive algebraic groups over local fields and adelic rings  (see Section \ref{sPadele}).
    
\end{enumerate}
\end{remark}

If $\pi$ is an atom in $\widehat{G}$, i.e., $\nu(\{\pi\})>0$,  
we can show $\pi$ is a discrete series representation and $\nu(\{\pi\})$ is just the formal dimension of $\pi$ \cite{DiCalg,Ro}. 
Under this assumption, if $G$ is a real Lie group that has discrete series and $\Gamma$ is an ICC group,  the theorem reduces to the special case of a single representation (see \cite{GHJ} Theorem 3.3.2)
\begin{center}
$\dim_{\mathcal{L}(\Gamma)}(H_{\pi})=\covol(\Gamma)\cdot d_{\pi}$.  
\end{center}
This is motivated by the study of discrete series of Lie groups by M. Atiyah and W. Schmid \cite{ASds77}.
(See also \cite{Ruthprod19} for a discussion of the $p$-adic $\GL(2)$ together with its cuspidal and Steinberg representations.) 

It is known that a Lie group $G$ has discrete series if and only if $\rk G=\rk K$ for a maximal compact subgroup $K$. 
Thus it excludes a large family of Lie groups, e.g, $\SL_n(\mathbb{R})$ ($n\geq 2$), $\SO(n,\mathbb{C})$, $\GL_n(\mathbb{C})$. 
A natural question is what is the analog for other Lie groups or other irreducible representations that are not discrete series. 
This is one of the motivations for this result. 
  

Before closing this section, I would like to introduce an example of $\SL(2,\mathbb{R})$ as an application of Theorem \ref{tdimmeas}. 
\begin{example}
Consider $G=\SL(2,\mathbb{R})$ and $\mathcal{H}=\{x+iy|x,y\in \mathbb{R},y>0\}$, the Poincar\'{e} upper-half plane with a $\SL(2,\mathbb{R})$-invariant measure $y^{-2}dxdy$. 
Observe that $\PSL(2,\mathbb{R})/\{\SO(2)/\{\pm 1\}\}\cong \mathcal{H}$. 
We fix a Haar measure on $\SO(2)/\{\pm 1\}$ with total measure $1$ and then take the Haar measure $\mu$ on $\PSL(2,\mathbb{R})$ to be the product measure on $\mathcal{H}\times \{\SO(2)/\{\pm 1\}\}$. 
Let $\nu$ be the Plancherel measure given by $\mu$. 

The classification of irreducible unitary representations of $\SL(2,\mathbb{R})$ is well-known (see \cite{Gel75}). 
However, the support of the Plancherel measure $\nu$ on the unitary dual $\widehat{\SL(2,\mathbb{R})}$ contains only the following two families of irreducible representations (see \cite{Fo2,Kn}). 
\begin{enumerate}
 
    \item {\it Discrete series representations}: $\{\pi_n^{\pm}|n\geq 2\}$. 
    The underlying space of $\pi_n^+$ is given as 
    \begin{center}
    $H_n^+=\{f\colon \mathcal{H}\to\mathbb{C}\text{ holomorphic}~|\int_{\mathcal{H}}|f(x+iy)|^2y^{n-2}dxdy<\infty\}$. 
    \end{center}
    For $g=\begin{pmatrix}
a & b\\
c & d
\end{pmatrix}\in \SL(2,\mathbb{R})$ and $f\in H_n^+$,  the action is given by $(\pi_k^+(g)f)(z)=(bz+d)^{-k}f\left(\frac{az+b}{cz+d}\right)$. 
The representation $\pi_n^-$ is defined to be the contragradient of $\pi_n^+$. 
We have $\nu(\{\pi_n^+\})=\nu(\{\pi_n^-\})=\frac{n-1}{4\pi}$ for $n\geq 2$. 
\item {\it Principal series representations}: $\{\pi^{\pm}_{it}|t>0\}$.  
The spaces $\{H_{it}^{\pm}\}_{t>0}$ for them are all $L^2(\mathbb{R})$ with distinct actions given by 
\begin{center}
$\pi^{+}_{it}\begin{pmatrix}
a & b\\
c & d
\end{pmatrix}f(x)=|bx+d|^{-1-it}f\left(\frac{ax+c}{bx+d}\right)$,\\
$\pi^{-}_{it}\begin{pmatrix}
a & b\\
c & d
\end{pmatrix}f(x)={\rm sgn}(bx+d)|bx+d|^{-1-it}f\left(\frac{ax+c}{bx+d}\right)$. 
\end{center}
Let $P=\{g=\begin{pmatrix}
a & b\\
0 & a^{-1}
\end{pmatrix}|a,b\in \mathbb{R}, a\neq 0\}$ with characters given by
$\chi_t^{\pm}\begin{pmatrix}
a & b\\
0 & a^{-1}
\end{pmatrix}=\varepsilon^{\pm}(a)|a|^{it}$, 
where $\varepsilon^+(a)=1$ and $\varepsilon^{-}(a)=\rm{sgn}(a)$.   
One can show $\pi^{\pm}_{it}=\Ind_P^G(\chi_t^{\pm})$. 
We have $d\nu(\pi_t^+)=\frac{t}{8\pi}\tanh{\frac{\pi t}{2}}dt$ and $d\nu(\pi_t^-)=\frac{t}{8\pi}\coth{\frac{\pi t}{2}}dt$.
\end{enumerate}

For the lattice $\PSL(2,\mathbb{Z})$, we have $\covol(\Gamma)=\frac{\pi}{3}$. 
By Theorem \ref{tdimmeas}, we list the dimensions of some modules over  $\mathcal{L}(\PSL(2,\mathbb{Z}))$ (which is a factor of type ${\rm II}_1$).

\begin{enumerate}
\item (Discrete part) Let $X=\{\pi_{n_1}^{+},\dots,\pi_{n_k}^{+}\}$. We have $H_X=\oplus_{1\leq i\leq k}\pi^+_{n_i}$. 
Hence

\centerline{ $\dim_{\mathcal{L}(\PSL(2,\mathbb{Z}))}H_X=\frac{\sum_{i=1}^{k}n_i}{12}-\frac{k}{12}$. }
Note the dimension above remains the same if we replace any $\pi_{n_i}^{+}$ by its dual $\pi_{n_i}^{-}$. 
    \item (Continuous part) Let $Y=\{\pi^{+}_{it}|t\in [a,b]\}$ for some $0<a<b$.  
    We have
    $H_{Y}=\int_{[a,b]}^{\oplus}H_{\pi_{it}^{+}}d\nu(\pi_{it}^{+})$. 
    Hence
    
    \centerline{$\dim_{\mathcal{L}(\PSL(2,\mathbb{Z}))}H_{Y}=\frac{1}{24}\int_a^b t\tanh{\frac{\pi t}{2}}dt$.}

\end{enumerate}

\end{example}


\section{The dimension and the adelic Plancherel measure}\label{smain}

Let $G$ be a reductive group defined over a number field $F$. 
Assume $G$ is of dimension $d$ over $F$ so that we are able to take a nonzero left-invariant $d$-form $\omega$ on $G$. 
For each infinite place $v\in V_{\infty}$, $\omega$ determines a left-invariant different form $\omega_v$ on the smooth manifold $G(F_v)$. 
This in turn gives a Haar measure $\mu_v$ of $G(F_v)$. 
For each finite place $v\in V_f$, let $p_v\in \mathbb{Z}$ be the prime number  lying under $v$. 
Let $\mu_v$ be the Haar measure on $G(F_v)$, which is uniquely determined by the corresponding $d$-form $\omega_v$ on the $p_v$-adic analytic manifold $G(F_v)$. 
The {\it Tamagawa measure}, denoted by $\mu_{G(\mathbb{A}_F)}$, or simply $\mu$, is defined to be the restricted product of the measures $\{\mu_v\}_{v\in V}$ on the adelic group $G(\mathbb{A}_F)=\prod_{v\in V_{\infty}}G(F_v)\times \prod'_{v\in V_f}G(F_v)$, which is independent of the choice of $\omega$ \cite{PR94}. 

The volume of $G(\mathbb{A}_F)/G(F)$ (if it exists) is called the {\it Tamagawa number} of $G$ and denoted by $\tau_F(G)$, or simply $\tau(G)$.  
In fact, for any finite extension of number fields $L/F$, we can prove $\tau_L(G)=\tau_F(G)$ \cite{Osttmgw84}. 
It suffices to consider the case $F=\mathbb{Q}$. 
The following result was conjectured by A. Weil and proved by R. Langlands \cite{Llds66Tama}, K. Lai \cite{Lai80Tama}, R. Kottwitz \cite{Kott88Tama} and V. Chernousov \cite{CnsvE8}. 

\begin{theorem}
\label{tWeilTama}
If $G$ is a simply connected semisimple group, then $\tau(G)=1$. 
\end{theorem}


By fixing the Tamagama measure $\mu$ on $G(\mathbb{A}_F)$ as a Haar measure and the corresponding Plancherel measure $\nu$ on $\widehat{G(\mathbb{A}_F)}$, we have the following result. 

\begin{theorem}\label{tmain1}
Let $G$ be a simply connected semisimple algebraic group over a number field $F$.  
Let $X\subset \widehat{G(\mathbb{A}_F)}$ such that $\nu(X)<\infty$ and $H_X=\int_{\pi\in X}^{\oplus}H_{\pi}d\nu_{G(\mathbb{A}_F)}(\pi)$. 
We have
\begin{center}
    $\dim_{\mathcal{L}(G(F))}H_X=\nu_{G(\mathbb{A}_F)}(X)$. 
\end{center}
\end{theorem}
\begin{proof}
By Theorem \ref{tdimmeas}, we have
\begin{center}
$\dim_{\mathcal{L}(G(F))}H_X=\covol(G(F))\cdot\nu_{G(\mathbb{A}_F)}(X)$. 
\end{center}
Then it follows Theorem \ref{tWeilTama}. 
\end{proof}

Here we give a criterion for the case when $\mathcal{L}(G(F))$ is a type $\text{II}_1$ factor, which is related to the local groups at some infinite places.
\begin{proposition}
If there is an infinite place $v$ such that $G(F_v)$ is center-free and connected,   
$G(F)$ is an ICC group and hence $\mathcal{L}(G(F))$ is a $\rm{II}_1$ factor. 
\end{proposition}
\begin{proof}
First, we observe $G(F)$ is dense in each local group $G(F_v)$ if the place $v$ is archimedean ($v\in V_{\infty}$). 

Take $\gamma\in G(F)$ and consider its conjugacy class $C_\gamma=\{g\gamma g^{-1}|g\in G(F)\}$ in $G(F)$. 
The map $\gamma\to C_{\gamma}$ defined on $G(F)$ can be extended to the group $G(F_v)$ by continuity. 
If $C_{\gamma}$ is finite, $C_{\gamma}=\overline{C_\gamma}=\{g\gamma g^{-1}|g\in G(F_v)\}$, which can be denoted as $Y=\{\gamma,g_{1}\gamma g_{1}^{-1},\dots,g_{N}\gamma g_{N}^{-1}\}$. 

Let $G(F_v)$ act on the set $Y$ by conjugation. 
The stabilizer of $\gamma$ is $Z_\gamma(G(F_v))=\{g\in G(F_v)|g\gamma g^{-1}=\gamma\}$. 
Note $N+1=|Y|=|G(F_v)x|=|G(F_v)/Z_\gamma(G(F_v))|$. 
That is to say, $Z_\gamma(G(F_v))$ is a closed subgroup of $G(F_v)$ of finite index. 
Therefore $Z_\gamma(G(F_v))=G(F_v)$ by the connectedness.
As $G(F_v)$ is center-free, $\gamma=e$ and $G(F)$ is an ICC group. 
\end{proof}

\bibliographystyle{abbrv}
\typeout{}
\bibliography{MyLibrary} 

\begin{thebibliography}{10}

\bibitem{APintrII1}
C.~Anantharaman and S.~Popa.
\newblock An introduction to $\text{II}_1$ factors.
\newblock {\em preprint}, 8, 2017.

\bibitem{ASds77}
M.~Atiyah and W.~Schmid.
\newblock A geometric construction of the discrete series for semisimple {L}ie
  groups.
\newblock {\em Invent. Math.}, 42:1--62, 1977.

\bibitem{Bek04}
B.~Bekka.
\newblock Square integrable representations, von {N}eumann algebras and an
  application to {G}abor analysis.
\newblock {\em J. Fourier Anal. Appl.}, 10(4):325--349, 2004.

\bibitem{Bnst74}
I.~N. Bern\v{s}te\u{\i}n.
\newblock All reductive {$p$}-adic groups are of type {I}.
\newblock {\em Funkcional. Anal. i Prilo\v{z}en.}, 8(2):3--6, 1974.

\bibitem{Blkd77}
B.~E. Blackadar.
\newblock The regular representation of restricted direct product groups.
\newblock {\em J. Functional Analysis}, 25(3):267--274, 1977.

\bibitem{CnsvE8}
V.~I. Chernousov.
\newblock The {H}asse principle for groups of type {$E_8$}.
\newblock {\em Dokl. Akad. Nauk SSSR}, 306(5):1059--1063, 1989.

\bibitem{Clz07}
L.~Clozel.
\newblock Spectral theory of automorphic forms.
\newblock In {\em Automorphic forms and applications}, volume~12 of {\em
  IAS/Park City Math. Ser.}, pages 43--93. Amer. Math. Soc., Providence, RI,
  2007.

\bibitem{DiCalg}
J.~Dixmier.
\newblock {\em {$C\sp*$}-algebras}.
\newblock North-Holland Mathematical Library, Vol. 15. North-Holland Publishing
  Co., Amsterdam-New York-Oxford, 1977.
\newblock Translated from the French by Francis Jellett.

\bibitem{Flath77}
D.~Flath.
\newblock Decomposition of representations into tensor products.
\newblock In {\em Automorphic forms, representations and {$L$}-functions
  ({P}roc. {S}ympos. {P}ure {M}ath., {O}regon {S}tate {U}niv., {C}orvallis,
  {O}re., 1977), {P}art 1}, Proc. Sympos. Pure Math., XXXIII, pages 179--183.
  Amer. Math. Soc., Providence, R.I., 1979.

\bibitem{Fo2}
G.~B. Folland.
\newblock {\em A course in abstract harmonic analysis}.
\newblock Studies in Advanced Mathematics. CRC Press, Boca Raton, FL, 1995.

\bibitem{Gel75}
S.~S. Gelbart.
\newblock {\em Automorphic forms on ad\`ele groups}.
\newblock Annals of Mathematics Studies, No. 83. Princeton University Press,
  Princeton, N.J.; University of Tokyo Press, Tokyo, 1975.

\bibitem{GHJ}
F.~M. Goodman, P.~de~la Harpe, and V.~F.~R. Jones.
\newblock {\em Coxeter graphs and towers of algebras}, volume~14 of {\em
  Mathematical Sciences Research Institute Publications}.
\newblock Springer-Verlag, New York, 1989.

\bibitem{J83}
V.~F.~R. Jones.
\newblock Index for subfactors.
\newblock {\em Invent. Math.}, 72(1):1--25, 1983.

\bibitem{J00ten}
V.~F.~R. Jones.
\newblock Ten problems.
\newblock In {\em Mathematics: frontiers and perspectives}, pages 79--91. Amer.
  Math. Soc., Providence, RI, 2000.

\bibitem{Kiri76}
A.~A. Kirillov.
\newblock {\em Elements of the theory of representations}.
\newblock Grundlehren der Mathematischen Wissenschaften, Band 220.
  Springer-Verlag, Berlin-New York, 1976.
\newblock Translated from the Russian by Edwin Hewitt.

\bibitem{Kiri04}
A.~A. Kirillov.
\newblock {\em Lectures on the orbit method}, volume~64 of {\em Graduate
  Studies in Mathematics}.
\newblock American Mathematical Society, Providence, RI, 2004.

\bibitem{Kn}
A.~W. Knapp.
\newblock {\em Representation theory of semisimple groups}, volume~36 of {\em
  Princeton Mathematical Series}.
\newblock Princeton University Press, Princeton, NJ, 1986.
\newblock An overview based on examples.

\bibitem{Kott88Tama}
R.~E. Kottwitz.
\newblock Tamagawa numbers.
\newblock {\em Ann. of Math. (2)}, 127(3):629--646, 1988.

\bibitem{Lai80Tama}
K.~F. Lai.
\newblock Tamagawa number of reductive algebraic groups.
\newblock {\em Compositio Math.}, 41(2):153--188, 1980.

\bibitem{Llds66Tama}
R.~P. Langlands.
\newblock The volume of the fundamental domain for some arithmetical subgroups
  of {C}hevalley groups.
\newblock In {\em Algebraic {G}roups and {D}iscontinuous {S}ubgroups ({P}roc.
  {S}ympos. {P}ure {M}ath., {B}oulder, {C}olo., 1965)}, pages 143--148. Amer.
  Math. Soc., Providence, R.I., 1966.

\bibitem{MvN36}
F.~J. Murray and J.~Von~Neumann.
\newblock On rings of operators.
\newblock {\em Ann. of Math. (2)}, 37(1):116--229, 1936.

\bibitem{Osttmgw84}
J.~Oesterl\'{e}.
\newblock Nombres de {T}amagawa et groupes unipotents en caract\'{e}ristique
  {$p$}.
\newblock {\em Invent. Math.}, 78(1):13--88, 1984.

\bibitem{PR94}
V.~Platonov and A.~Rapinchuk.
\newblock {\em Algebraic groups and number theory}, volume 139 of {\em Pure and
  Applied Mathematics}.
\newblock Academic Press, Inc., Boston, MA, 1994.
\newblock Translated from the 1991 Russian original by Rachel Rowen.

\bibitem{Ro}
A.~Robert.
\newblock {\em Introduction to the representation theory of compact and locally
  compact groups}, volume~80 of {\em London Mathematical Society Lecture Note
  Series}.
\newblock Cambridge University Press, Cambridge-New York, 1983.

\bibitem{Ruthprod19}
L.~C. Ruth.
\newblock The product of lattice covolume and discrete series formal dimension:
  $p$-adic {$\rm GL(2)$}.
\newblock {\em arXiv preprint arXiv:1901.11501}, 2019.

\end{thebibliography}

\textit{E-mail address}: \href{mailto:junyang@fas.harvard.edu}{junyang@fas.harvard.edu}\\

{Harvard University, Cambridge, MA 02138, USA}

\end{document}